\documentclass[11pt,headings=small]{scrartcl}

\usepackage[sumlimits]{amsmath}
\usepackage{amssymb,amsfonts,amsthm}
\usepackage[english]{babel}
\usepackage[latin1]{inputenc}
\usepackage[T1]{fontenc}
\usepackage{lmodern}
\usepackage{graphicx}
\usepackage{scrpage2}
\usepackage{hyperref}
\usepackage{color}
\usepackage{setspace}
\usepackage{textcomp}
\usepackage{enumerate}
\usepackage{xfrac}
\usepackage{helvet}
\usepackage{bbm}
\usepackage{multirow}
\usepackage{bigstrut}
\usepackage{booktabs}

\newtheorem{thm}{Theorem}[section]
\newtheorem{prp}[thm]{Proposition}
\newtheorem{cor}[thm]{Corollary}
\newtheorem{lem}[thm]{Lemma}
\theoremstyle{definition}

\def\P{\mathcal{P}}
\def\R{\mathbb{R}}

\def\N{\mathbb{N}}

\def\F{\mathbb{F}}

\def\D{\mathbb{D}}

\def\1{\mathbbm{1}}

\DeclareMathOperator{\bmo}{BMO}
\DeclareMathOperator{\e}{e}

\allowdisplaybreaks[1]
\clearscrheadfoot
\ihead{\headmark}
\ohead{\thepage}

\begin{document}
\pagestyle{scrheadings}
\onehalfspacing

\title{BMO and  exponential Orlicz space estimates of the discrepancy function in arbitrary dimension}
\author{Dmitriy Bilyk$^a$, Lev Markhasin$^b$\\\\
$^a$ \scriptsize School of Mathematics, University of Minnesota, 206 Church St. SE, Minneapolis, MN, 55408, USA \\ \scriptsize email:  dbilyk@math.umn.edu\\
$^b$ \scriptsize Institut f\"ur Stochastik und Anwendungen, Universit\"at Stuttgart, Pfaffenwaldring 57, 70569 Stuttgart, Germany \\ \scriptsize email: lev.markhasin@mathematik.uni-stuttgart.de}
\maketitle

\begin{abstract}
In the current paper we 
obtain   discrepancy estimates  in  exponential Orlicz  and $\bmo$ spaces  in arbitrary dimension $d \ge 3$. In particular, we use  dyadic harmonic analysis to  prove that  the dyadic product  $\bmo$ and $\exp \big( L^{2/(d-1)} \big)$ norms of the discrepancy function of    so-called digital nets of order two are bounded above by $(\log N)^{\frac{d-1}{2}}$.  The latter bound has been recently conjectured in several papers and is consistent with the best known low-discrepancy constructions. Such estimates play an important role as an intermediate step between the well-understood $L_p$ bounds and the notorious open problem of finding the precise $L_\infty$ asymptotics of the discrepancy function in higher dimensions, which is still elusive.

\end{abstract}

\noindent{\footnotesize {\it 2010 Mathematics Subject Classification.} Primary 11K06,11K38,42C10,46E35,65C05. } \\[5mm]

\section{Introduction and results}
\subsection{Definitions}

The main object of the present paper is the {\it{discrepancy function}}. For a positive integer $N$ let $\P_N$ be a point set in the unit interval $[0,1)^d$ with $N$ points. The discrepancy function is defined as
\begin{align*}
D_{\P_N}(x) = \sum_{z \in \P_N} \chi_{[0,x)}(z) - N x_1 \cdots x_d
\end{align*}
where $x = (x_1, \ldots, x_d) \in [0,1)^d$ and $[0,x) = [0,x_1)\times\ldots\times[0,x_d)$. By $\chi_A$ we denote the characteristic function of a set $A\in\R^d$, so the term $C_{\P_N}(x) = \sum_z \chi_{[0,x)}(z)$ is equal to the number of points of $\P_N$ in the interval $[0,x)$. Hence, $D_{\P_N}$ measures the deviation of the number of points of $\P_N$ in $[0,x)$ from the fair number of points $L_N(x) = N |[0,x)| = N \, x_1 \cdots x_d$, which would be achieved by a (practically impossible) perfectly uniform distribution of points, thus quantifying the extent of equidistribution of the point set $\P_N$ and its quality for numerical integration (quasi-Monte Carlo methods, see e.g. \cite{DP10}). 

Asymptotic behavior of the discrepancy function in $L_p([0,1)^d)$-spaces for $1<p<\infty$ is well understood. 
The classical lower bound  proved by Roth \cite{R54} for $p=2$ and by Schmidt \cite{S77} for arbitrary $1<p<\infty$ states that there exists a constant $c = c(p,d) > 0$ such that for every positive integer $N$ and all point sets $\P_N$ in $[0,1)^d$ with $N$ points, we have
\begin{align} \label{UpperRoth}
\left\|D_{\P_N}|L_p([0,1)^d)\right\|\geq c\,\left(\log N\right)^{(d-1)/2}.
\end{align}
The best known value for $c$ in $L_2$ can be found in \cite{HM11}. Furthermore, these estimates are known to be sharp, i.e. there exists a constant $C = C(p,d) > 0$ such that for every positive integer $N$, there is a point set $\P_N$ in $[0,1)^d$ with $N$ points such that
\begin{align}\label{UpperLp}
\left\|D_{\P_N}|L_p([0,1)^d)\right\|\leq C\,\left(\log N\right)^{(d-1)/2}.
\end{align}
This was proved by Davenport \cite{D56} for $p=2,d=2$, by Roth \cite{R80} for $p=2$ and arbitrary $d$, and finally by Chen \cite{C80} in the general case. The best known value for $C$ in $L_2$ can be found in \cite{DP10} and \cite{FPPS10}.

The precise asymptotics  of  the $L_{\infty}([0,1)^d)$-norm of the discrepancy function   (star-discrepancy) is  known as {\emph{the great open problem in discrepancy theory} \cite{BC87}. 
The best currently known lower bound in dimensions $d\ge 3$ was obtained quite recently \cite{BLV08}. There exists a constant $c  = c(d) > 0$ such that for every positive integer $N$ and all point sets $\P_N$ in $[0,1)^d$ with $N$ points, we have
\begin{align}
\left\|D_{\P_N}|L_\infty([0,1)^d)\right\|\geq c\,\left(\log N\right)^{(d-1)/2+\eta_d}
\end{align}
where $0 < \eta_d < 1/2$. At the same time,   the bound in the plane is well  known (\cite{S72})
\begin{align}\label{schmidt}
\left\|D_{\P_N}|L_\infty([0,1)^2)\right\|\geq c\,\log N.
\end{align}
Furthermore (e.g., \cite{Hl60}), there exists a constant $C > 0$ such that for every positive integer $N$, there is a point set $\P_N$ in $[0,1)^d$ with $N$ points such that
\begin{align}\label{Linfty}
\left\|D_{\P_N}|L_\infty([0,1)^d)\right\|\leq C\,\left(\log N\right)^{d-1}.
\end{align}
One can observe a gap between the known upper and lower bounds for the star discrepancy in dimensions $d\ge 3$.  There is no agreement among the experts as to what should be the correct asymptotics in higher dimension, the two main conjectures being $(\log N)^{d-1}$ and $(\log N)^{d/2}$. We refer the reader e.g. to \cite{B11} for a more detailed discussion.

\subsection{Main results}

Since the precise behavior of discrepancy in   $L_p$-spaces ($1<p<\infty$) is known, while the $L_\infty$ estimates remain elusive, it is natural and  instructive to investigate what happens in intermediate spaces ``close'' to $L_\infty$. Standard examples of such spaces are the exponential Orlicz spaces and (various versions of)  BMO, which stands for {\emph{bounded mean oscillation}}. In harmonic analysis, these spaces often play a role of a natural substitute for $L_\infty$ as an endpoint of the $L_p$ scale. We refer the reader to the next section for precise definitions and references. 

This approach was initiated in \cite{BLPV09} in the case of dimension $d=2$. 
Examples used to prove upper bounds in 
the two-dimensional case were constructed as modifications of the celebrated Van der Corput set. In higher dimensions we resort to the  higher-order digital nets -- a concept introduced by Dick \cite{D07}, \cite{D08} and studied from the relevant point of view in \cite{DP14a}, \cite{D14}, and \cite{M14}. In particular, we strongly  rely on the estimates of the Haar coefficients of the discrepancy function for such nets (see Lemma \ref{lem_ord2})  recently obtained by the second author \cite{M14}.

The first  main result of this work is the  upper bound in  dyadic product BMO.

\begin{thm} \label{main_result}
 For any dimension $d\ge 3$ there exists a constant $C = C(d)  > 0$ such that for every positive integer $N$, there is a point set $\P_N$ in $[0,1)^d$ with $N$ points such that
\begin{align}\label{bmo_upper}
 \left\|D_{\P_N}|\bmo^d\right\|\leq C\,\left(\log N\right)^{(d-1)/2}.
\end{align}
\end{thm}

This result is known in the plane (see \cite[Theorem 1.7]{BLPV09}), moreover, it is sharp. A simple modification of the proof of \eqref{UpperRoth} yields the corresponding lower bound.
\begin{thm}\label{bmo_lower}
For any dimension $d\ge 3$ there exists a constant $c = c(d) >0$ such that for every positive integer $N$ and all point sets $\P_N$ in $[0,1)^d$ with $N$ points we have
\begin{align}
\left\|D_{\P_N}|\bmo^d\right\|\geq c\,\left(\log N\right)^{(d-1)/2}.
\end{align}
\end{thm}
These results say that in the case of discrepancy function, the $\bmo$ norm behaves more like $L_p$ rather than like $L_\infty$.\\

Furthermore, we  extend the result of \cite[Theorem 1.4]{BLPV09}   on  exponential Orlicz spaces to the case of arbitrary dimension. The main theorem we prove in this direction is the following.

\begin{thm} \label{main_result_exp_upper}
 In any dimension $d\ge 3$, there exists a constant $C = C(d)  > 0$ such that for every positive integer $N$, there is a point set $\P_N$ in $[0,1)^d$ with $N$ points, for which 
\begin{align}\label{exp_upper}
 \left\|D_{\P_N} \bigg| \exp \Big(L^{\frac{2}{d-1}}   \Big)\right\|\leq C\,\left(\log N\right)^{\frac{ d-1}{2}}.
\end{align}
\end{thm}

Some remarks are in order for this theorem. This bound has been recently conjectured in several different sources. A similar (albeit weaker)  estimate has been recently proved  for the so-called Chen--Skriganov nets independently in \cite{S13} and \cite{ABLV14}  for the smaller $\exp \big( L^{2/(d+1)} \big)$ norm. The authors of both papers conjectured that it should be  improved to the  $\exp \big(L^{2/(d-1)} \big)$ estimate stated above, in addition, the same conjecture has been made in the survey paper \cite[Section 9]{DP14b}. However, until now this claim remained unproved. 

The exponential integrability exponent $2/(d-1)$ is quite natural for a variety of reasons. First, it is consistent with the general ideology that the problem effectively has $d-1$ ``free parameters'' (see \cite{B11} for a detailed discussion) and therefore the Littlewood--Paley inequalities should be applied $d-1$ times: see \S\ref{s.LP}, in particular, estimate \eqref{e.LPh} of Lemma \ref{LPd}. Furthermore, this estimate is consistent with the $L_\infty$-discrepancy bound \eqref{Linfty} of the order $(\log N)^{d-1}$ valid for digital nets, see \S\ref{remarks}. 

For this very reason the complementary lower bound is presently beyond reach. There are indications that it should be almost as difficult as one of the main open problems in the subject -- the lower bound of the $L_\infty$-discrepancy. The proof of the corresponding lower bound  in dimension $d=2$ (\cite[Theorem 1.4]{BLPV09}) uses techniques similar to the proof of the two-dimensional $L_\infty$ bound \eqref{schmidt}, which are not available in higher dimensions. Besides, if one believes that the correct $L_\infty$ bound is $(\log N)^{d/2}$, then estimate \eqref{exp_upper} will probably not be sharp -- in this case the norm on the left-hand side should be the subgaussian $\exp (L^2)$, see \S\ref{remarks} for details.

During the final stages of  preparation of the present manuscript, we have learned about a recent preprint of Skriganov \cite{S14} written almost simultaneously, where inequality \eqref{exp_upper} is proved for {\emph{random}} digit shifts of an arbitrary digital $(t,n,d)$-net (although the author doesn't state the result in exponential form, but instead writes down $L_p$ estimates with explicit dependence  on $p$). The techniques of Skriganov's work  exploit randomness in a crucial way. In contrast, our proof is deterministic and is applicable to any higher order digital net (in fact, it suffices to take order $\sigma = 2$). Concrete construction of such nets are given, e.g., in \cite{D14}.

Interpolating the estimate of Theorem \ref{main_result_exp_upper} with the well-known $L_\infty$ bound  \eqref{Linfty} we obtain the following result, which is a direct analog of \cite[Theorem 1.4]{BLPV09}.
\begin{cor} \label{cor_exp_upper}
 For each $\beta$ satisfying  $\frac{2}{d-1}\leq \beta<\infty$, there exists a constant $C_\beta > 0$ such that for every positive integer $N$, there is a point set $\P_N$ in $[0,1)^d$ with $N$ points such that
\begin{align}\label{exp_upper1}
 \left\|D_{\P_N}|\exp(L^\beta)\right\|\leq C_\beta \,\left(\log N\right)^{(d-1)-\frac1{\beta}}.
\end{align}
\end{cor}
Since this result is even more closely tied to the $L_\infty$ estimates, no corresponding lower bounds are available.

Our strategy resonates with  that of \cite{BLPV09}, but we also strongly rely on very recent results  and constructions: digital nets of higher order \cite{D07}, \cite{D08} and their explicit constructions   \cite{D14}, \cite{DP14b},  Haar expansions of the discrepancy function of such nets used in the study of discrepancy in Besov spaces with dominating mixed smoothness and in $L_2$, see \cite{DP14a}, \cite{DP14b}, \cite{D14}, \cite{M14}. For further results on this topic see \cite{CS02}, \cite{CS08}, \cite{Hi10}, \cite{M13a}, \cite{M13b}, \cite{S06}, \cite{T10}.
As general references for studies of the discrepancy function we refer to the monographs \cite{BC87}, \cite{DP10}, \cite{KN74}, \cite{M99}, \cite{NW10}  and surveys \cite{B11}, \cite{Hi14}, \cite{M13c}.

We shall write $A \preceq B$ if there exists an absolute  constant $c>0$ such that $A \leq c\,B$. We write $A \simeq B$ if $A\preceq B $ and $B \preceq A$. The implicit constants in this paper do not depend on the number of   points $N$ (but may depend on some other parameters, such as dimension, integrability index etc).

\section{Preliminary facts}
\subsection{Haar bases}

We denote $\N_{-1}=\N_0\cup\{-1\}$. Let $\D_j = \{0,1,\ldots, 2^j-1\}$ for $j \in \N_0$ and $\D_{-1} = \{0\}$. For $j = (j_1,\dots,j_d)\in\N_{-1}^d$ let $\D_j = \D_{j_1}\times\ldots\times \D_{j_d}$. For $j\in\N_{-1}^d$ we write $|j| = \max(j_1,0) + \ldots + \max(j_d,0)$.

For $j \in \N_0$ and $m \in \D_j$ we call the interval $I_{j,m} = \big[ 2^{-j} m, 2^{-j} (m+1) \big)$ the $m$-th dyadic interval in $[0,1)$ on level $j$. We put $I_{-1,0}=[0,1)$ and call it the $0$-th dyadic interval in $[0,1)$ on level $-1$. Let $I_{j,m}^+ = I_{j + 1,2m}$ and $I_{j,m}^- = I_{j + 1,2m+1}$ be the left and right half of $I_{j,m}$, respectively. 

For $j \in \N_{-1}^d$ and $m = (m_1, \ldots, m_d) \in \D_j$ we call $I_{j,m} = I_{j_1,m_1} \times \ldots \times I_{j_d,m_d}$ the $m$-th dyadic interval in $[0,1)^d$ at level $j$. We call the number $|j|$ the order of the dyadic interval $I_{j,m}$. Its volume is then $| I_{j,m} | = 2^{-|j|}$.


An important combinatorial fact is that $\# \{ j \in \mathbb N_0^d: |j| = n \} \simeq n^{d-1}$, where $\#$ stands for the cardinality of a set.

Let $j \in \N_{0}$ and $m \in \D_j$. Let $h_{j,m}$ be the function on $[0,1)$ with support in $I_{j,m}$ and  constant values $1$ on $I_{j,m}^+$ and $-1$ on $I_{j,m}^-$. We put $h_{-1,0} = \chi_{I_{-1,0}}$ on $[0,1)$. Notice that we normalize the Haar functions in $L_\infty$, rather than $L_2$.

Let $j \in \N_{-1}^d$ and $m \in \D_j$. The function $h_{j,m}$ given as the tensor product
\[ h_{j,m}(x) = h_{j_1,m_1}(x_1) \cdots h_{j_d,m_d}(x_d) \]
for $x = (x_1, \ldots, x_d) \in [0,1)^d$ is called a dyadic Haar function on $[0,1)^d$. The set of functions $\{h_{j,m}:\; j \in \N_{-1}^d, \, m \in \D_j\}$ is called dyadic Haar basis on $[0,1)^d$. 

It is well known that the system
\[ \left\{2^{\frac{|j|}{2}}h_{j,m} \,:\,j\in\N_{-1}^d,\,m\in \D_j\right\} \]
is an orthonormal basis in $L_2([0,1)^d)$, an unconditional basis in $L_p([0,1)^d)$ for $1 < p < \infty$, and a conditional basis in $L_1([0,1)^d)$. 

\subsection{Littlewood--Paley inequalities}\label{s.LP}

For any function $f\in L_2([0,1)^d)$ we have Parseval's identity
\begin{equation}\label{parseval}
 \left\|f|L_2([0,1)^d)\right\|^2 = \sum_{j \in \N_{-1}^d} 2^{|j|} \sum_{m\in \D_j}|\langle f,h_{j,m}\rangle|^2.  
 \end{equation}
Littlewood--Paley inequalities are a generalization of this statement to $L_p$-spaces.  For a function $f:[0,1]^d \rightarrow \mathbb R$, the (dyadic) Littlewood--Paley square function is defined as  
\begin{equation*}
Sf (x)  =   \Bigg(\sum_{\substack{j\in\N_{-1}^d }} 2^{2|j|} \sum_{m\in \D_j} |\langle f,h_{j,m}\rangle|^2 \, \chi_{I_{j,m}}\Bigg)^{1/2}.
\end{equation*}
It is a classical  fact and a natural extension of \eqref{parseval} that in dimension $d=1$, the $L_p$-norm of $f$ can be characterized using the square function, i.e. for each $1<p<\infty$ there exist constants $A_p$, $B_p >0$ such that 
\begin{equation}\label{LP1D}
A_p \| Sf|L_p([0,1)) \| \le \| f |L_p([0,1)) \| \le B_p \| Sf |L_p([0,1)) \|.
\end{equation}
Two remarks are important. First, it is well known that $B_p \simeq  \sqrt{p}$. Second, estimates \eqref{LP1D} continue to hold for Hilbert space-valued functions $f$. This  allows one to  extend the inequalities  to the case of multivariate functions $f:[0,1]^d \rightarrow \mathbb R$ by iterating the one-dimensional estimates $d$ times, thus picking up constants $A_p^d$ and $B_p^d \simeq p^{d/2}$. 

However, if the function $f$ is represented by a {\emph{hyperbolic}} sum of Haar wavelets, i.e. a sum of Haar functions supported by intervals of fixed order, $f \in \textup{span} \{ h_{j,m}:\, |j| = n \}$, in other words, when the number of ``free parameters'' is $d-1$, then the one-dimensional Littlewood--Paley inequalities \eqref{LP1D} only need to be applied $d-1$ times, yielding  constants $A_p^{d-1}$ and $B_p^{d-1} \simeq  p^{\frac{d-1}{2}}$. We summarize  the estimates useful for our purposes in the following lemma.
\begin{lem}\label{LPd}
Let $1<p<\infty$. 
\begin{enumerate}[(i)]
\item {\emph{Multiparameter Littlewood--Paley inequality:}} For a function $f: [0,1]^d \rightarrow  \mathbb R$ we have 
\begin{equation}\label{e.LPd}
  \left\|f| \, L_p([0,1)^d)\right\| \preceq p^{\frac{d}{2}}   \left\| Sf | \, L_p([0,1)^d)\right\|.
\end{equation}
\item {\emph{Hyperbolic Littlewood--Paley inequality:}} Assume that  $f: [0,1]^d \rightarrow  \mathbb R$ is a hyperbolic sum of Haar functions, i.e. $f \in \textup{span} \{ h_{j,m}:\, |j| = n \}$ for some $n \in \mathbb N$. Then
\begin{equation}\label{e.LPh}
  \left\|f| \, L_p([0,1)^d)\right\| \preceq p^{\frac{d-1}{2}}   \left\| Sf | \, L_p([0,1)^d)\right\|.
\end{equation}
\end{enumerate}
\end{lem}

A more detailed discussion of the Littlewood--Paley inequalities and their applications in discrepancy theory can be found in \cite{B11}.

\subsection{Bounded mean oscillation and exponential Orlicz spaces}

There are  different definitions of the space of functions of bounded mean oscillation in the multivariate case. The appropriate version in our setting is the so-called {\emph{dyadic product}}  $\bmo^d$ introduced in \cite{Be}. For an integrable function $f:[0,1]^d \rightarrow \mathbb R$ we define 
\begin{equation}\label{def_bmo}
 \left\|f|\bmo^d\right\| = \sup_{U\subset [0,1)^d} \left( |U|^{-1} \sum_{j \in \N_0^d} 2^{|j|} \sum_{\substack{m\in \D_j \\ I_{j,m} \subset U}} |\langle f,h_{j,m}\rangle|^2 \right)^{1/2}, 
 \end{equation}
where the supremum is taken over all measurable sets $U\subset [0,1)^d$. The space $\bmo^d$ contains all  integrable functions $f$ with  finite norm $\left\|f|\bmo^d\right\|$. Notice that technically $ \left\|f|\bmo^d\right\|$ is only a seminorm, since it vanishes on linear combinations of functions, which are constant in some of the coordinate directions, therefore formally we need to take a factor space over such functions.

To give some intuition behind this definition, we notice that when $d=1$ and $U$ is a dyadic interval, we have by Parseval's identity $$\displaystyle{ |U|^{-1} \sum_{j \in \N_0} 2^{|j|} \sum_{\substack{m\in \D_j \\ I_{j,m} \subset U}} |\langle f,h_{j,m}\rangle|^2 = |U|^{-1} \int_U \Big| f - \langle f\rangle_U \Big|^2 dx},$$ where $\langle f\rangle_U$ is the mean of $f$ over $U$ -- this is precisely the expression which arises in the definition of the one-dimensional dyadic $\bmo$ space. The precise technical definition of the norm \eqref{def_bmo} turns out to be the correct multiparameter dyadic extension, which preserves the most natural properties  of $\bmo$, in particular, the celebrated $H_1 - \bmo$ duality: {\emph{dyadic product BMO}} is the dual of the {\emph{dyadic}} Hardy space $H_1$ -- the space of functions $f\in L_1$ with integrable Littlewood--Paley square function, i.e. such that $Sf \in L_1$, see \cite{Be}. \\

We remark that a non-dyadic version of this space, the {\emph{Chang--Fefferman product}} BMO was introduced and studied in \cite{CF80}. This space also admits a characterization similar to \eqref{def_bmo}, but with smoother functions in  place of Haar wavelets. For the relation between these spaces see e.g. \cite{PW}.\\

In order to introduce  the definition of the exponential Orlicz spaces, we start by briefly discussing general Orlicz spaces. We refer to \cite{LT77} for more information. Let $(\Omega,P)$ be a probability space and let $\mathbb{E}$ denote the expectation over $(\Omega,P)$. Let $\psi:\,[0,\infty)\rightarrow[0,\infty)$ be a convex function, such that $\psi(x) = 0$ if and only if $x=0$. For a $(\Omega,P)$-measurable real valued function $f$ we define
\[ \left\|f|L^\psi\right\| = \inf\{K>0:\, \mathbb{E}\psi(|f|/K)\leq 1\}, \]
where $\inf\emptyset = \infty$. The Orlicz space (associated with $\psi$) $L^\psi$ consists of all functions $f$ with finite norm $\left\|f|L^\psi\right\|$.

Let $\alpha>0$ and  let $\psi_\alpha$  be a convex function which equals $\e^{x^\alpha} - 1$ for $x$ sufficiently large, depending upon $\alpha$ (for $\alpha \ge 1$ this function may be used for all $x\ge0$). We denote $\exp(L^\alpha)=L^{\psi_\alpha}$. 

The following proposition yields a standard way to compute the $\exp (L^\alpha) $ norms. Its proof is a simple application of Taylor's series for $e^x$ and Stirling's formula.

\begin{prp}\label{exp_Lp}
For any $\alpha >0$, the following equivalence holds
\begin{equation}\label{e.exp_Lp}
\left\|f|\exp(L^\alpha)\right\|  \simeq \, \sup_{p>1} \, p^{-\frac1{\alpha}} \cdot \left\| f | L_p ([0,1)^d)\right\|.
\end{equation}
\end{prp}

The next proposition is a variant of the famous Chang--Wilson--Wolff inequality \cite{CWW85} which states that  boundedness of the square function implies certain exponential integrability of the original function. The hyperbolic version  presented here can be easily deduced from the  Littlewood--Paley inequality with sharp constants \eqref{e.LPh} and the previous proposition.

\begin{prp}[Hyperbolic Chang--Wilson--Wolff inequality]\label{changwilsonwolff}
Assume that $f$ is a hyperbolic sum of multiparameter  Haar functions, i.e. $f \in \textup{span} \{ h_{j,m}:\, |j| = n \}$    for some $n \in \mathbb N$. Then  
\begin{equation}\label{cww}
\left\|f|\exp \big(L^{2/(d-1) } \big)\right\| \preceq \left\|S(f)|L_\infty([0,1)^d)\right\|. 
\end{equation}
\end{prp}

\begin{proof}
According to \eqref{e.LPh}, we have $\| f|L_p([0,1)^d) \| \preceq p^{\frac{d-1}2} \| Sf |L_p([0,1)^d) \| \linebreak \le p^{\frac{d-1}2} \| Sf|L_\infty([0,1)^d) \|$. Estimate \eqref{cww} now follows from \eqref{e.exp_Lp}.
\end{proof}

We note that it is important here that the function $f$ is a linear combination of Haar functions supported by rectangles  of fixed volume: without this assumption the correct norm in the left-hand side would have been $\exp \big( L^{2/d} \big)$  which can be deduced from \eqref{e.LPd}.

For all $1\leq p <\infty$  we have $L_\infty \subset \exp (L^\alpha) \subset L_p$. Furthermore, it is obvious that  $\left\|f|\exp(L^\alpha)\right\| \preceq \left\|f|\exp(L^\beta)\right\|$, i.e. $\exp (L^\beta) \subset \exp (L^\alpha) $, for $\alpha <\beta$.  The next lemma shows that if we assume that $f \in L_\infty$, the relation may be reversed. The argument  is a simple interpolation between exponential Orlicz spaces and $L_\infty$.

\begin{prp}\label{interpolationexp}
 Let $0 <  \alpha < \beta <\infty$. Consider a function $f\in L_\infty([0,1)^d)$. If $f\in\exp(L^\alpha)$, then also  $f\in\exp(L^\beta)$ and we have
\[ \left\|f|\exp(L^\beta)\right\| \preceq \Big\|f|\exp(L^\alpha)\Big\|^{\alpha/\beta} \cdot \left\|f|L_\infty([0,1)^d)\right\|^{1-\alpha/\beta}. \]
\end{prp}

\begin{proof} Set $q = \frac{\alpha}{\beta} p$. Then we have $\| f |L_p([0,1)^d) \| \le \|f |L_q([0,1)^d) \|^{\alpha/\beta}  \cdot  \| f |L_\infty([0,1)^d) \|^{1 - \alpha/\beta}$ and 
\begin{align*}
 \left\|f|\exp(L^\beta)\right\|  & \simeq \, \sup_{p>1} \, p^{-\frac{1}{\beta}} \cdot  \| f |L_p([0,1)^d) \| \le \sup_{p>1} \, p^{-\frac1{\beta}} \cdot  \| f |L_q([0,1)^d) \|^{\alpha/\beta} \cdot  \| f |L_\infty([0,1)^d) \|^{1 - \alpha/\beta}\\
 & \preceq \sup_{q>1} \left(  q^{-\frac1{\alpha}} \cdot  \| f |L_q([0,1)^d) \| \right)^{\alpha/\beta} \cdot  \| f |L_\infty([0,1)^d) \|^{1 - \alpha/\beta},
\end{align*} 
which finishes the proof. 
\end{proof}

\subsection{Digital nets}

Our next step is to define digital $(t,n,d)$-nets of order $\sigma\ge 1$. The original definition of digital nets  goes back to Niederreiter \cite{N87}, and the first constructions were given even earlier by Sobol' \cite{S67}. The concept of higher-order digital nets was introduced in \cite{D07}, \cite{D08}. We quote the definitions  from \cite{D07} and \cite[Definitions 4.1, 4.3]{D08}. In the case of order $\sigma = 1$,  the  original definition of digital nets is recovered.

For $n,\sigma\in\N$ let $C_1,\ldots,C_d$ be $\sigma n\times n$ matrices over $\F_2$. For $\nu\in\{0,1,\ldots,2^n-1\}$ with the binary expansion $\nu = \nu_0 + \nu_1 2 + \ldots + \nu_{n-1} 2^{n-1}$ with digits $\nu_0,\nu_1,\ldots,\nu_{n-1}\in\{0,1\}$, the binary digit vector $\bar{\nu}$ is given as $\bar{\nu} = (\nu_0,\nu_1,\ldots,\nu_{n-1})^{\top}\in\F_2^n$. Then we compute $C_i\bar{\nu}=(x_{i,\nu,1},x_{i,\nu,2},\ldots,x_{i,\nu,\sigma n})^{\top}\in\F_2^{\sigma n}$ for $1\leq i\leq d$. Finally we define
\[ x_{i,\nu}=x_{i,\nu,1} 2^{-1}+x_{i,\nu,2} 2^{-2}+\ldots+x_{i,\nu,\sigma n} 2^{-\sigma n} \in[0,1) \]
and $x_{\nu}=(x_{1,\nu},\ldots,x_{d,\nu})$. We call the point set $\P_{2^n}=\{x_0,x_1,\ldots,x_{2^n-1}\}$ a digital net (over $\F_2$).

Now let $0\leq t\leq \sigma n$ be an integer. For every $1\leq i\leq d$ we write $C_i=(c_{i,1},\ldots,c_{i,\sigma n})^{\top}$ where $c_{i,1},\ldots,c_{i,\sigma n}\in\F_2^n$ are the row vectors of $C_i$. If for all $1\leq \lambda_{i,1}<\ldots<\lambda_{i,\eta_i}\leq \sigma n,\,1\leq i\leq d$ with
\[ \lambda_{1,1}+\ldots+\lambda_{1,\min(\eta_1,\sigma)}+\ldots+\lambda_{d,1}+\ldots+\lambda_{d,\min(\eta_d,\sigma)}\leq\sigma n - t \]
the vectors $c_{1,\lambda_{1,1}},\ldots,c_{1,\lambda_{1,\eta_1}},\ldots,c_{d,\lambda_{d,1}},\ldots,c_{d,\lambda_{d,\eta_d}}$ are linearly independent over $\F_2$, then $\P_{2^n}$ is called an order $\sigma$ digital $(t,n,d)$-net (over $\F_2$).

The smaller the quality parameter $t$ and the greater the order $\sigma$, the better structure the point set has. In particular every point set $\P_{2^n}$ constructed with the digital method is at least an order $\sigma$ digital $(\sigma n,n,d)$-net. Every order $\sigma_2$ digital $(t,n,d)$-net is an order $\sigma_1$ digital $(\lceil t\sigma_1/\sigma_2 \rceil,n,d)$-net if $1\leq\sigma_1\leq\sigma_2$ (see \cite{D07}).  It is well known that digital $(t,n,d)$-nets  are perfectly distributed with respect to dyadic intervals (in the standard terminology, see e.g. \cite{DP10}, order $1$ digital $(t,n,d)$-nets are $(t,n,d)$-nets): {\emph{every dyadic interval of order $n-t$ contains exactly $2^t$ points of the $(t,n,d)$-net.}} A version of this property continues to hold for higher-order nets.

\begin{lem}\label{digital_net}
 Let $\P_{2^n}$ be an order $\sigma$ digital $(t,n,d)$-net, then every dyadic interval of order $n$ contains at most $2^{\lceil t/\sigma \rceil}$ points of $\P_{2^n}$.
\end{lem}
It is a classical fact that such sets satisfy the best known star discrepancy estimate \eqref{Linfty}, see \cite[Theorem 5.10]{DP10}.
\begin{lem}\label{net_discr}
Let $\P_{2^n}$ be an order $\sigma$ digital $(t,n,d)$-net, then 
\begin{equation}\label{e.net_discr}
\left\|D_{\P_{2^n}}|L_\infty([0,1)^d)\right\| \preceq n^{d-1}.
\end{equation}
\end{lem}

Constructions of order $\sigma$ digital $(t_2,n,d)$-nets can be obtained via so-called digit interlacing of order $1$ digital $(t_1,n,\sigma d)$-nets and several constructions of order $1$ digital nets are known. For details, examples and further literature we refer to \cite{DP14b} and \cite{D14}. We only point out here that there are constructions with a good quality parameter $t$, which in particular does not depend on $n$.

\section{Proofs of the theorems}
We will prove the main theorems in the case when the number of points is a power of two, i.e. $N=2^n$. The reduction to the general case is standard, see e.g. \S6.3 in \cite{BLPV09}. Our examples are  the higher-order digital nets described in the previous section with the minimal non-trivial value of the order $\sigma=2$.

We shall rely on the recent estimates of the Haar coefficients of the discrepancy function of   order $2$ digital nets obtained by the second author. The following result is \cite[Lemma 5.9]{M14}.

\begin{lem}\label{lem_ord2}
 Let $\P_{2^n}$ be an order $2$ digital $(t,n,d)$-net. Let $j\in\N_{-1}^d$ and $m\in \D_j$.
\begin{enumerate}[(i)]
 \item If $|j|\geq n-\lceil t/2\rceil$, then $|\langle D_{\P_{2^n}},h_{j,m}\rangle|\preceq 2^{-|j|}$ and $|\langle D_{\P_{2^n}},h_{j,m}\rangle|\preceq 2^{-2|j| +n}$ for all but $2^n$ values of $m$.\label{lem_ord2_part1}
 \item If $|j|< n-\lceil t/2\rceil$, then $|\langle D_{\P_{2^n}},h_{j,m}\rangle|\preceq 2^{-n}\left(2n-t-2|j|\right)^{d-1}$.\label{lem_ord2_part2}
\end{enumerate}                                                                                                 
\end{lem}

In fact, we shall mostly need the second part of this  lemma, i.e. the Haar coefficients  for small values of $|j|$ (in other words, for large intervals). If we compare this estimate to the corresponding two-dimensional bound for the Van der Corput set obtained in \cite[Lemma 4.1]{BLPV09}, which stated that $|\langle D_{\P_{2^n}},h_{j,m}\rangle|\preceq 2^{-n}$, we see that in our case we have an additional logarithmic factor. It is however completely harmless as one can see from the following elementary computation.

\begin{lem}\label{geom_ser}
 Let $K$ be a positive integer, $A>1$, and $q,r > 0$. Then we have
\[ \sum_{k=0}^{K-1} A^k\,(K-k)^q\,k^r \preceq A^K\,K^r, \] where the implicit constant is independent of $K$.
\end{lem}

\begin{proof}
We have
 \begin{align*}
  \sum_{k=0}^{K-1} A^k\,(K-k)^q\,k^r &\leq A^K\,K^r \sum_{k=0}^{K-1} A^{k-K}\,(K-k)^q 
                                     = A^K\,K^r \sum_{k=1}^{K} A^{-k}\,k^q \preceq A^K\,K^r.
 \end{align*}
\end{proof}

We now turn to the proofs of the main theorems, which are similar in spirit to the arguments in \cite{BLPV09}.

\subsection{Proof of Theorem \ref{main_result}}

Let $\P_{2^n}$ be an order $2$ digital $(t,n,d)$-net with the quality parameter $t $ depending only on the dimension $d$. 
We recall that  $\# \{ j \in \mathbb N_0^d: |j| = n \} \simeq n^{d-1}$ and $\#\D_j=2^{|j|}$. We fix an arbitrary measurable set $U \subset [0,1)^d$. We need to prove
\[ |U|^{-1} \sum_{j \in \N_0^d} 2^{|j|} \sum_{\substack{m\in \D_j \\ I_{j,m} \subset U}} |\langle D_{\P_{2^n}},h_{j,m}\rangle|^2 \preceq n^{d-1}. \]
We split the sum above  into three cases: large, intermediate, and small intervals, according to the cases in Lemma \ref{lem_ord2}. We observe that, in each case,  there are at most $2^{|j|}|U|$ values of $m \in \D_j$ such that $I_{j,m} \subset U$. 

Starting with large intervals, we apply \eqref{lem_ord2_part2} of Lemma \ref{lem_ord2} and Lemma \ref{geom_ser} to obtain
\begin{align*}
 & |U|^{-1} \sum_{\substack{j\in\N_0^d \\ |j|< n-\lceil t/2\rceil}} 2^{|j|} \sum_{\substack{m\in \D_j \\ I_{j,m} \subset U}} |\langle D_{\P_{2^n}},h_{j,m}\rangle|^2 \\
 &\preceq |U|^{-1} \sum_{\substack{j\in\N_0^d \\ |j|< n-\lceil t/2\rceil}} 2^{|j|}\,2^{|j|}|U|\,2^{-2n}\left(2n-t-2|j|\right)^{2(d-1)} \\
 &\preceq 2^{-2n} \sum_{k=0}^{n-\lceil t/2\rceil -1} 2^{2k}\left(2n-t-2k\right)^{2(d-1)}(k+1)^{d-1} \\
 &\preceq 2^{-2n}\,2^{2n}\, (n-t/2)^{d-1} \preceq n^{d-1}.
\end{align*}

Next, we consider  intermediate intervals and apply \eqref{lem_ord2_part1} of Lemma \ref{lem_ord2} to obtain
\begin{align*}
 &|U|^{-1} \sum_{\substack{j\in\N_0^d \\ n-\lceil t/2\rceil \leq |j| < n}}2^{|j|}\sum_{\substack{m\in \D_j \\ I_{j,m} \subset U}} |\langle D_{\P_{2^n}},h_{j,m}\rangle|^2 \\
 &\preceq |U|^{-1} \sum_{\substack{j\in\N_0^d \\ n-\lceil t/2\rceil \leq |j| < n}}2^{|j|}\,2^{|j|}|U|\,2^{-2|j|}\\
 &\leq \sum_{k=n-\lceil t/2\rceil}^{n-1}(k+1)^{d-1} \preceq n^{d-1}.
\end{align*}

We now turn to the case of small intervals, where $|j| \ge n$.  These boxes are too small to capture any  cancellation, hence we will treat the linear and counting parts of the discrepancy function separately. The case of the linear part $L_{\mathcal P_{2^n}}(x) = 2^n  x_1 \cdot \ldots \cdot x_d$ is simple. It is easy to verify that $|\langle L_{\mathcal P_{2^n}},h_{j,m}\rangle| \simeq 2^{-2|j| +n}$, thus we obtain
\begin{align*}
 & |U|^{-1} \sum_{\substack{j\in\N_0^d \\ |j|\geq n}} 2^{|j|} \sum_{\substack{m\in \D_j \\ I_{j,m} \subset U}} |\langle L_{\mathcal P_{2^n}} ,h_{j,m}\rangle|^2 \\ 
 &\preceq |U|^{-1} \sum_{\substack{j\in\N_0^d \\ |j|\geq n}} 2^{|j|}\,2^{|j|}|U|\,2^{-4|j| +2n} \\
 &\preceq  2^{2n} \sum_{k=n}^\infty 2^{-2k}\,(k+1)^{d-1} \preceq n^{d-1}.
\end{align*}

Estimating the counting part $C_{\P_{2^n}}$ is a bit more involved. Let $\mathcal J$ denote the family of all dyadic  intervals  $I_{j,m } \subset U$ with $|j| \ge n $ and  such that $\langle C_{\mathcal P_{2^n} } , h_{j,m}  \rangle \neq 0$. Consider the subfamily $\widetilde{\mathcal J} \subset \mathcal J$, which consists of {\emph{maximal}} (with respect to inclusion) dyadic intervals in $\mathcal J$. We first demonstrate the following fact, which provides control of the total size of the intervals in this family
\begin{equation}\label{e.maxu}
\sum_{I_{j,m} \in \widetilde{\mathcal J}}  | I_{j,m} |  = \sum_{I_{j,m} \in \widetilde{\mathcal J}} 2^{-|j|} \preceq  n^{d-1} |U|.
\end{equation}

Indeed, consider an interval  $I_{j,m} \in \mathcal J$. Since $\langle C_{\mathcal P_{2^n} } , h_{I_{j,m} } \rangle \neq 0$, this implies that at least one point $z \in \mathcal P_{2^n}$ must be contained in the interior of $I_{j,m}$, which in turn means that each side of $I_{j,m}$ has length at least $2^{-2n}$ (since $\mathcal P_{2^n}$ is an order $2$ digital net, whose points  have binary coordinates of length $2n$), i.e. $0\le j_k \le 2n$ for $k=1,\ldots,d$.

Fix integer parameters $r_1$,\ldots,$r_{d-1}$ between $0$ and $2n$. Consider the family $\widetilde{\mathcal J}_{r_1,\ldots,r_{d-1}} \subset \widetilde{\mathcal J}$ consisting of those intervals $I_{j,m} \in \widetilde{\mathcal J}$ for which $j_k = r_k$ for $k=1,\ldots, d-1$, i.e. the lengths of their first $d-1$ sides are fixed. Then all intervals in this family are disjoint: if two of them intersected, then their first $d-1$ sides would have to coincide, and hence one would have to be contained in the other, which would contradict maximality. Therefore, we find that 
\begin{align*}
\sum_{I_{j,m} \in \widetilde{\mathcal J}} 2^{-|j|} &=  \sum_{r_1,\ldots,r_{d-1} = 0}^{2n} \sum_{I_{j,m} \in \widetilde{\mathcal J}_{r_1,\ldots,r_{d-1}}  }  \big| I_{j,m} \big|  
 \le  \sum_{r_1,\ldots,r_{d-1} = 0}^{2n}  |U|  \preceq  n^{d-1} |U|,
\end{align*}
which proves \eqref{e.maxu}.

For a dyadic interval $J$, we define
\[ C^J_{\mathcal P_{2^n}} (x) = \sum_{z\in \mathcal P_{2^n} \cap J} \chi_{[0,x)} (z), \]
i.e. the part of the counting function, which counts only the  points from $J$. It is clear that $\langle C_{\mathcal P_{2^n} }, h_{j,m} \rangle = \langle C^J_{\mathcal P_{2^n} }, h_{j,m} \rangle$ whenever $I_{j,m} \subset J$.

We recall Lemma \ref{digital_net} which implies that any dyadic interval of volume at most $2^{-n}$ contains no more than $2^{\lceil t/2 \rceil}$ points. Therefore, for any interval $J \in \widetilde{\mathcal J}$ we have 
\begin{equation}\label{e.countL2}
\big\| C^J_{\mathcal P_{2^n}} \big\|_{L_2 ( J) } \le \sum_{p\in \mathcal P_{2^n} \cap J}  \big\| \chi_{[0,\cdot)}(z) \big\|_{L_2(J)} \leq 2^{\lceil t/2 \rceil} |J|^{\frac12}.
\end{equation}
Using  the orthogonality of Haar functions, Bessel inequality, \eqref{e.countL2}, and \eqref{e.maxu}, we find that
\begin{align*}
  |U|^{-1} \sum_{\substack{I_{j,m} \in \mathcal J }}  2^{|j|} |\langle C_{\mathcal P_{2^n} }, h_{j,m} \rangle|^2 & \le |U|^{-1} \sum_{J \in \widetilde{\mathcal J}}  \sum_{I_{j,m} \subset J }  2^{|j|} |\langle C^J_{\mathcal P_{2^n} }, h_{j,m} \rangle|^2 \\ & \le |U|^{-1} \sum_{J \in \widetilde{\mathcal J}} \big\| C^J_{\mathcal P_{2^n}} \big\|_{L_2 ( J) }^2  \leq  |U|^{-1} \, 2^{t+1} \sum_{J \in \widetilde{\mathcal J}} |J|   \preceq  n^{d-1},
\end{align*}
which concludes the proof for small intervals and therefore proves Theorem \ref{main_result}.


\subsection{Proof of Theorem \ref{bmo_lower}}

We now turn to the  proof of the matching lower bound for the space $\bmo^d$. The proof is a simple adaptation of the ideas of the original proof \cite{R54} of the lower bound for the $L_2$-discrepancy \eqref{UpperRoth}. Fix an arbitrary point set $\mathcal P_N \subset [0,1)^d$ with $N$ points. Choose the scale $n \in \mathbb N$ so that $ 2N \le 2^n < 4N$. This choice guarantees that for  each $j \in \mathbb N_0^d$ with $| j | = n$, there are at least $2^{n-1}$ values of $m \in \D_j$ such that $I_{j,m} \cap \mathcal P_N = \emptyset$, i.e. at least half of all intervals do not contain any  points of $\mathcal P_N$. As discussed before, for such empty intervals $|\langle D_{\mathcal P_N},h_{j,m}\rangle| = |\langle L_N,h_{j,m}\rangle| \simeq N 2^{-2|j|}  \simeq 2^{-n}$. We use the definition of the $\bmo^d$ norm  \eqref{def_bmo} and choose the measurable set $U= [0,1)^d$ to obtain
\begin{align*}
\left\| D_{\mathcal P_N} |\bmo^d\right\|^2 & \ge  \sum_{j \in \N_0^d} 2^{|j|} \sum_{\substack{m\in \D_j }} |\langle D_{\mathcal P_N} ,h_{j,m}\rangle|^2  \ge   \sum_{\substack{j \in \N_0^d \\ | j | = n}} 2^{|j|} \sum_{\substack{m\in \D_j \\ I_{j,m} \cap \mathcal P_N = \emptyset }} |\langle L_N ,h_{j,m}\rangle|^2 \\
& \succeq \sum_{\substack{j \in \N_0^d \\ | j | = n}} 2^n \cdot 2^{n-1} \cdot 2^{-2n} \simeq  n^{d-1},
\end{align*}
which finishes the proof, since $n \simeq \log N$. 

\subsection{Proof of Theorem \ref{main_result_exp_upper}}

 
We now turn our attention to the proof of the upper bound in the Orlicz space $\exp \big( L^{2/(d-1)} \big)$.  
Once again we consider three different cases, namely large, intermediate, and small intervals. 

We start with the large intervals. Applying the triangle inequality,  Chang--Wilson--Wolff inequality (Proposition \ref{changwilsonwolff}),  and part \eqref{lem_ord2_part2} of Lemma \ref{lem_ord2}, 
we obtain
\begin{align*}
 &\Bigg\|\sum_{\substack{j\in\N_{-1}^d \\ |j|< n-\lceil t/2\rceil}} 2^{|j|} \sum_{m\in \D_j} \langle D_{\P_{2^n}},h_{j,m}\rangle \, h_{j,m}\, \big|\exp \big(L^{2/(d-1) } \big) \Bigg\|  \\
 & \le \sum_{k=0}^{n-\lceil t/2\rceil} \Bigg\|\sum_{\substack{j\in\N_{-1}^d \\ |j|=k}} 2^{|j|} \sum_{m\in \D_j} \langle D_{\P_{2^n}},h_{j,m}\rangle \, h_{j,m} \, \big|\exp \big(L^{2/(d-1)} \big)  \Bigg\|  \\
 &\preceq \sum_{k=0}^{n-\lceil t/2\rceil} \Bigg\|\Bigg(\sum_{\substack{j\in\N_{-1}^d \\ |j|=k}} 2^{2|j|} \sum_{m\in \D_j} |\langle D_{\P_{2^n}},h_{j,m}\rangle|^2 \, \chi_{I_{j,m}}\Bigg)^{1/2}\, \big| \, L_\infty \Bigg\| \\
 &\preceq \sum_{k=0}^{n-\lceil t/2\rceil} \Bigg\|\Bigg(\sum_{\substack{j\in\N_{-1}^d \\ |j|=k}} 2^{2k}\,2^{-2n}\left(2n-t-2|j|\right)^{2(d-1)} \sum_{m\in \D_j} \chi_{I_{j,m}}\Bigg)^{1/2} \, \big| \, {L_\infty} \Bigg\| \\
 &\preceq 2^{-n} \sum_{k=0}^{n-\lceil t/2 \rceil }  \left( 2^{2k}\left(2n-t-2k\right)^{2(d-1)}(k+1)^{d-1}\right)^{1/2} \\
 &\preceq 2^{-n}\,2^{n}n^{(d-1)/2}  = n^{(d-1)/2}.
\end{align*}
Now we consider the medium sized intervals applying \eqref{lem_ord2_part1} of Lemma \ref{lem_ord2} and obtaining
\begin{align*}
 &\Bigg\|\sum_{\substack{j\in\N_{-1}^d \\ n-\lceil t/2\rceil \leq |j| < n}}2^{|j|}\sum_{m\in \D_j} \langle D_{\P_{2^n}},h_{j,m}\rangle\, h_{j,m}\,  \big|\exp \big(L^{2/(d-1)} \big)  \Bigg\|  \\
 & \le \sum_{k=n-\lceil t/2\rceil}^{n-1}  \Bigg\|\sum_{\substack{j\in\N_{-1}^d:\, |j | = k }}2^{|j|}\sum_{m\in \D_j} \langle D_{\P_{2^n}},h_{j,m}\rangle\, h_{j,m} \, \big|\exp \big(L^{2/(d-1)} \big)  \Bigg\|  \\
 &\leq   \sum_{k=n-\lceil t/2\rceil}^{n-1} \Bigg\| \Bigg(\sum_{\substack{j\in\N_{-1}^d:\, | j |= k }}2^{2k}\sum_{m\in \D_j} |\langle D_{\P_{2^n}},h_{j,m}\rangle|^2\, \chi_{I_{j,m}}\Bigg)^{1/2} \, \big| \, L_\infty \Bigg\|  \\
 &\preceq   \sum_{k=n-\lceil t/2\rceil}^{n-1}  \left((k+1)^{d-1}\right)^{1/2} \preceq n^{(d-1)/2}.
\end{align*}

In the case of small intervals we  again treat the linear and the counting parts separately. Since $|\langle L_{\mathcal P_{2^n}},h_{j,m}\rangle| \preceq 2^{-2|j| +n}$ we obtain
\begin{align*}
 & \Bigg\|\sum_{\substack{j\in\N_{-1}^d \\ |j|\geq n}} 2^{|j|} \sum_{m\in \D_j} \langle L_{\mathcal P_{2^n}},h_{j,m}\rangle\, h_{j,m}\, \big|\exp \big(L^{2/(d-1)} \big) \Bigg\|  \\
& \le  \sum_{k=n}^\infty \Bigg\|\sum_{\substack{j\in\N_{-1}^d:\, |j|=k}} 2^{|j|} \sum_{m\in \D_j} \langle L_{\mathcal P_{2^n}},h_{j,m}\rangle\, h_{j,m} \, \big|\exp \big(L^{2/(d-1)} \big)  \Bigg\|  \\
 &\leq \sum_{k=n}^\infty  \Bigg\| \Bigg(\sum_{\substack{j\in\N_{-1}^d :\, |j|=k }} 2^{2k} \sum_{m\in \D_j} |\langle L_{\mathcal P_{2^n}},h_{j,m}\rangle|^2\, \chi_{I_{j,m}}   \Bigg)^{1/2} \, \big| \, L_\infty   \Bigg\|  \\
 &\preceq 2^{n}  \sum_{k=n}^\infty \left( \,  2^{-2k}\,(k+1)^{d-1}\right)^{1/2} \preceq n^{(d-1)/2}.
\end{align*}

The estimate of the  counting part is somewhat harder.  Recall that  $\mathcal J$ denotes the family of all dyadic  intervals  $I_{j,m } \subset U$ with $|j| \ge n $, i.e. $| I_{j,m} | \le 2^{-n}$,  such that $\langle C_{\mathcal P_{2^n} } , h_{j,m}  \rangle \neq 0$.  As noticed earlier, if $I_{j,m} \in \mathcal J$, this implies that $I_{j,m}$ contains at least one point of $\mathcal P_{2^n}$ in its interior and therefore $j_k \le 2n$ (i.e. $ | I_{j_k,m_k} | \ge 2^{-2n}$) for each $k=1,\ldots,d$.  

In addition, for each $I_{j,m} \in \mathcal J$, we can find its  unique {\emph{parent}} $\widetilde{I}_{j',m'}$ which satisfies the following conditions: (i) $I_{j,m} \subset \widetilde{I}_{j',m'}$;  (ii) $|j'| = 2^{-n}$, i.e. $| \widetilde{I}_{j',m'} | = 2^{-n}$; and (iii) $j_k = j'_k$ (which implies that $I_{j_k,m_k}=  \widetilde{I}_{j'_k,m'_k}$) for all $k=1,\ldots,d-1$. In other words, to find the parent, we expand the $d$-th side of $I_{j,m}$ so that the resulting interval has volume $2^{-n}$. We can now reorganize the sum according to the parents  
\begin{align}\label{e.sum1} 
 &  \sum_{\substack{j\in\N_{-1}^d \\ |j|\geq n}} 2^{|j|} \sum_{m\in \D_j} \langle C_{\mathcal P_{2^n}},h_{j,m}\rangle\, h_{j,m}   =      \sum_{\substack{   \widetilde{I}_{j',m'}:\, |j'|=n \\ j'_k \le 2n:\, k=1,\ldots,d}}    
\,\, \sum_{\substack{I_{j,m} \subset  \widetilde{I}_{j',m'} \\ j_k = j'_k:\, k=1,\ldots, d-1}}  2^{|j|}  \langle C_{\mathcal P_{2^n}},h_{j,m}\rangle\, h_{j,m}  .
\end{align}
Fix an arbitrary   parent interval $ \widetilde{I}_{j',m'}$ and consider the innermost sum above
\begin{align}\label{e.sum2}
\sum_{\substack{I_{j,m} \subset  \widetilde{I}_{j',m'} \\ j_k = j'_k:\, k\le d-1}}  2^{|j|}  \langle C_{\mathcal P_{2^n}},h_{j,m}\rangle\, h_{j,m}  = \sum_{p  \in \mathcal P_{2^n} \cap  \widetilde{I}_{j',m'}} 
\sum_{\substack{I_{j,m} \subset  \widetilde{I}_{j',m'} \\ j_k = j'_k:\, k\le d-1}}  2^{|j|}  \langle {{\chi}}_{[p,1)},h_{j,m}\rangle\, h_{j,m}. 
\end{align}

\noindent We notice that the expression inside the last sum splits into a product of one-dimensional factors:
\begin{align*}
2^{|j|}  \langle {{\chi}}_{[p,1)},h_{j,m}\rangle\, & h_{j,m}  (x)  = \prod_{j=1}^d 2^{j_k}   \langle {{\chi}}_{[p_k,1)},h_{j_k,m_k}\rangle\, h_{j_k,m_k } (x_k)\\
& = \Bigg( \prod_{j=1}^{d-1} 2^{j'_k}   \langle {{\chi}}_{[p_k,1)},h_{j'_k,m'_k}\rangle\, h_{j'_k,m'_k } (x_k) \Bigg)\cdot  2^{j_d}   \langle {{\chi}}_{[p_d,1)},h_{j_d,m_d}\rangle\, h_{j_d,m_d } (x_d)\\
& =  2^{| j _*'|}   \langle {{\chi}}_{[p_*,1)},h_{j_*',m_*'}\rangle\, h_{j_*',m_*' } (x_1,\ldots,x_{d-1}) \cdot 2^{j_d}   \langle {{\chi}}_{[p_d,1)},h_{j_d,m_d}\rangle\, h_{j_d,m_d } (x_d),
\end{align*}
where by $*$ we have denoted the projection of the $d$-dimensional vector to its first $d-1$ coordinates, e.g. if $j=(j_1,\ldots,j_d)$, then $j_* = (j_1,\ldots,j_{d-1})$. The expression in \eqref{e.sum2} can now be rewritten as
\begin{align*}
& \sum_{\substack{I_{j,m} \subset  \widetilde{I}_{j',m'} \\ j_k = j'_k:\, k\le d-1}}  2^{|j|}   \langle {{\chi}}_{[p,1)},h_{j,m}\rangle\, h_{j,m} (x) \\ & =  2^{| j _*'|}   \langle {{\chi}}_{[p_*,1)},h_{j_*',m_*'}\rangle\, h_{j_*',m_*' } (x_*) \cdot \sum_{\substack{I_{j_d,m_d} \subset  \widetilde{I}_{j'_d,m'_d} }}  2^{j_d}  \langle {{\chi}}_{[p_d,1)},h_{j_d,m_d}\rangle\, h_{j_d,m_d} (x_d),
\end{align*}
leaving us with the task to examine this ultimate one-dimensional sum. However, one can easily see that this sum is precisely  the Haar expansion of the function ${{\chi}}_{[p_d,1)}$ restricted to the interval $\widetilde{I}_{j'_d,m'_d}$, except for the constant term, i.e. 
\begin{equation}\label{e.sum3} \sum_{\substack{I_{j_d,m_d} \subset  \widetilde{I}_{j'_d,m'_d} }}  2^{j_d}  \langle {{\chi}}_{[p_d,1)},h_{j_d,m_d}  \rangle\, h_{j_d,m_d} (x_d) = \chi_{\widetilde{I}_{j'_d,m'_d}} (x_d) \cdot \Big(  \chi_{[p_d,1)} (x_d)  -  2^{j_d} \Big| [p_d,1) \cap \widetilde{I}_{j'_d,m'_d}\Big|   \Big),
\end{equation}
 which, in particular, is bounded pointwise by $2$.  Obviously, $\Big|  2^{| j _*'|}   \langle {{\chi}}_{[p_*,1)},h_{j_*',m_*'}\rangle \Big| \le 1$. We recall that, according to Lemma \ref{digital_net},  there are at most $2^{\lceil t/2 \rceil}$ points $p \in \mathcal P_{2^n} \cap  \widetilde{I}_{j',m'}$. Therefore,
  \begin{align*}
\,\, \sum_{\substack{I_{j,m} \subset  \widetilde{I}_{j',m'} \\ j_k = j'_k:\, k=1,\ldots, d-1}}  2^{|j|}  \langle C_{\mathcal P_{2^n}},h_{j,m}\rangle\, h_{j,m} (x) = 
\alpha_{j'_d} (x_d) h_{j_*',m_*' } (x_*), 
 \end{align*}
where  $|\alpha_{j'_d} (x_d) | \le 2^{{\lceil t/2 \rceil}+1} \preceq 1$. 

Let $x_d$ be fixed for the moment. Due to \eqref{e.sum3}, for a given $(d-1)$-dimensional dyadic interval $\widetilde{I}_{j'_*,m'_*}$, there exists at most one  $d$-dimensional dyadic interval $\widetilde{I}_{j',m'} = \widetilde{I}_{j'_*,m'_*} \times \widetilde{I}_{j'_d,m'_d}$ with $ | \widetilde{I}_{j',m'} | = 2^{-n}$ and such that $\alpha_{j'_d} (x_d) \neq 0$. Therefore,  applying \eqref{e.sum1} and taking $L_p$-norms in the first $d-1$ variables we obtain:
\begin{align*}\label{e.sum4} 
 & \Bigg\|\sum_{\substack{j\in\N_{-1}^d \\ |j|\geq n}} 2^{|j|} \sum_{m\in \D_j} \langle C_{\mathcal P_{2^n}},h_{j,m}\rangle\, h_{j,m} \, \big| \, {L_p (dx_*)} \Bigg\|  \\ 
 & =   \Bigg\|  \sum_{\substack{   \widetilde{I}_{j',m'}:\, |j'|=n  \\ j'_k \le 2n:\, k=1,\ldots,d}}    
\,\, \sum_{\substack{I_{j,m} \subset  \widetilde{I}_{j',m'} \\ j_k = j'_k:\, k=1,\ldots, d-1}}  2^{|j|}  \langle C_{\mathcal P_{2^n}},h_{j,m}\rangle\, h_{j,m}  \, \big| \, {L_p (dx_*)}  \Bigg\|  \\
 & =      \Bigg\|  \sum_{\substack{   \widetilde{I}_{j'_*,m'_*}:\, |j'|=n \\ j'_k \le 2n,\, k=1,\ldots, d-1}}  \alpha_{j'_d} (x_d) h_{j_*',m_*' }   \, \big| \, {L_p (dx_*)}  \Bigg\|  \\
 & \preceq p^{\frac{d-1}{2}}  \Bigg\|  \Bigg( \sum_{\substack{   \widetilde{I}_{j'_*,m'_*}:\, |j'|=n \\ j'_k \le 2n,\, k=1,\ldots, d-1}}  |\alpha_{j'_d} (x_d)|^2 \chi_{\widetilde{I}_{j'_*,m'_*}}  \Bigg)^{1/2}  \, \big| \, {L_p (dx_*)}  \Bigg\|  \preceq p^{\frac{d-1}{2}} n^{\frac{d-1}{2}},
  \end{align*}
  where in the last line we have employed the $(d-1)$-dimensional Littlewood--Paley inequality (\ref{e.LPd}) and the fact that there are of the order of $n^{d-1}$ choices of $j'_*$ in the sum. Integrating this bound with respect to $x_d$ and applying Proposition \ref{exp_Lp} we arrive at 
\begin{equation*}
\Bigg\|\sum_{\substack{j\in\N_{-1}^d \\ |j|\geq n}} 2^{|j|} \sum_{m\in \D_j} \langle C_{\mathcal P_{2^n}},h_{j,m}\rangle\, h_{j,m} \, \big| {\exp \big(L^{2/(d-1)} \big)} \Bigg\| \preceq n^{\frac{d-1}{2}},
\end{equation*}
which finishes the proof of Theorem \ref{main_result_exp_upper}.


\subsection{Proof of Corollary \ref{cor_exp_upper}}

We set $\alpha = \frac2{d-1}$ and use Proposition \ref{interpolationexp} to interpolate between the $\exp \big(L^{2/(d-1)} \big)$ estimate \eqref{exp_upper} of Theorem \ref{main_result_exp_upper} and the $L_\infty$ estimate \eqref{e.net_discr} of Lemma \ref{net_discr}:
\begin{align*}
 \left\|D_{\P_{2^n}}|\exp(L^\beta)\right\| &\preceq \left\|D_{\P_{2^n}}|\exp(L^{2/(d-1)})\right\|^{\frac{2}{(d-1)\beta}} \cdot \left\|D_{\P_{2^n}}|L_\infty([0,1)^d)\right\|^{1-\frac{2}{(d-1)\beta}} \\
 &\preceq n^{\frac{d-1}{2} \cdot \frac{2}{(d-1)\beta} } \, n^{(d-1)\cdot \left(1-   \frac{2}{(d-1)\beta} \right)}  = n^{(d-1)-\frac1{\beta}}.
\end{align*}

\subsection{Orlicz space estimates and star-discrepancy}\label{remarks}
 
 In the end we would like to outline an argument which demonstrates how  estimates in exponential Orlicz spaces may be related to the ``great open problem'' of the subject \cite{BC87}, i.e. sharp bounds on the $L_\infty$-discrepancy. Let us assume that for a certain order $2$ digital net $\mathcal P_{2^n}$ with $N=2^n$ points, for some $\alpha>0$ the discrepancy function satisfies an exponential bound
 \begin{equation}\label{ifexp}
 \left\| D_{\mathcal P_{2^n}} | \, \exp (L^\alpha) \right\| \preceq (\log N)^{\frac{d-1}{2}} \simeq n^{\frac{d-1}{2}}.
 \end{equation}
 This trivially leads to the following distributional estimate: for each $\lambda >0$
 \[ \mu \left\{ x\in [0,1]^d : \, \big| D_{\mathcal P_{2^n}} (x) \big| > \lambda \right\} \preceq \exp \left( - \bigg(\frac{\lambda}{n^{(d-1)/2}}\bigg)^\alpha  \right),
\] where $\mu $ is the Lebesgue measure. The fact that $\mathcal P_{2^n}$ is a binary digital net (i.e. all points have binary coordinates of length $2n$) implies that its discrepancy function does not change much on dyadic intervals of side length $2^{-2n}$. Therefore, for those  values of $\lambda$, for which the set $\big\{ \big| D_{\mathcal P_{2^n}} (x) \big| > \lambda \big\}$ is non-empty, we must have 
\[  \mu \left\{ x\in [0,1]^d : \, \big| D_{\mathcal P_{2^n}} (x) \big| > \lambda \right\} \succeq 2^{-2nd} .\]
Comparing the last two estimates, we observe that they cannot simultaneously hold, if $\lambda \succeq n^{\frac{d-1}2 + \frac{1}{\alpha}}$, i.e. in this case the set $\big\{ \big| D_{\mathcal P_{2^n}} (x) \big| > \lambda \big\} = \emptyset$, in other words
\begin{equation}\label{exptostar}
 \big\| D_{\mathcal P_{2^n}}   \big\|_{L_\infty } \preceq n^{\frac{d-1}2 + \frac{1}{\alpha}}. 
 \end{equation}
Recall that two main conjectures about the correct asymptotics of the discrepancy function predict  the sharp order of growth of either $(\log N)^{d-1}$ or $(\log N)^{d/2}$. Our Theorem \ref{main_result_exp_upper} is consistent with the first hypothesis: in this case  \eqref{ifexp} holds with $\alpha = \frac{2}{d-1}$ and hence \eqref{exptostar} becomes $\big\| D_{\mathcal P_{2^n}}   \big\|_{L_\infty } \preceq n^{d-1}$ which matches the best known upper bound \eqref{Linfty}.

If one strives to prove the second conjecture along these lines, estimate \eqref{ifexp} should hold with $\alpha =2$, i.e. one would need to construct a digital net whose discrepancy function is subgaussian.

We notice that Skriganov \cite[Lemma 6.2]{S14}  uses a somewhat different discretization approach which yields similar results and shows that an estimate in Orlicz space  $\operatorname{exp} \big(L^{2/(d-1)} \big)$ yields the $L_\infty$ upper bound for the discrepancy function of the order $(\log N)^{d-1}$.

\subsection{Acknowledgements} Both authors would like to express gratitude to the organizers of the workshop ``Discrepancy, Numerical Integration and Hyperbolic Cross Approximation'' (HCM, Bonn, Germany, 2013), the international conference  MCQMC 2014 (KU Leuven, Belgium), the Hausdorff trimester program ``Harmonic Analysis and PDE'' (HIM, Bonn, Germany, 2014), the semester program ``High-dimensional approximation'' (ICERM, Brown University, Providence, RI, USA, 2014), where they had an opportunity to meet and discuss their work. The  first author was supported by the NSF grant DMS-1260516.

\end{document}